\documentclass[10pt,a4paper,reqno]{amsart}
\usepackage{amsfonts,amsthm,latexsym,amsmath,amssymb,amscd,amsmath,
epsf, 
}
\usepackage{graphicx}
\input cyracc.def
\newfam\cyrfam
\font\tencyr=wncyr10
\font\sevencyr=wncyr7

\textfont\cyrfam=\tencyr \scriptfont\cyrfam=\sevencyr
  \scriptscriptfont\cyrfam=\sevencyr

\font\tencyr=wncyr10

\font\tencyri=wncyi10
\font\eightcyr=wncyr8

\newtheorem{Def}{Definition}

\newtheorem{Prop}{Proposition}

\newtheorem{theo+}           {Theorem}
\newtheorem{prop+}           {Proposition}
\newtheorem{coro+}           {Corollary}
\newtheorem{lemm+}           {Lemma}

\theoremstyle{definition}
\newtheorem{defi+}           {Definition}

\newtheorem{example}           {Example}

\theoremstyle{remark}
\newtheorem{rema+}           {Remark}

\newenvironment{theorem}{\begin{theo+}}{\end{theo+}}
\newenvironment{proposition}{\begin{prop+}}{\end{prop+}}
\newenvironment{corollary}{\begin{coro+}}{\end{coro+}}
\newenvironment{lemma}{\begin{lemm+}}{\end{lemm+}}

 \def\Int{\mathbb{Z}}

\newcommand {\CC}{\mathcal C}

\newcommand {\bCP} {\mathbb {CP}}

\newcommand{\bC}{\mathbb C}
\newcommand{\bR}{\mathbb R}

\newcommand{\C}{\mathcal C}

\def\CP{{{\mathbb C}{\rm P}}}

\def\d{{\partial}}
\def\db{{\bar{\partial}}}

\def\newop#1{\expandafter\def\csname #1\endcsname{\mathop{\rm
#1}\nolimits}}

\newop{slc}
\newop{sign}
\newop{mdeg}

\begin{document}
          \numberwithin{equation}{section}

          \title[Level functions of quadratic  differentials  and the Strebel property]
          {Level functions of quadratic  differentials,  signed measures, and the Strebel property}
          
          \author[Y.~Baryshnikov]{Yuliy Baryshnikov}
\address{
Department of Mathematics, University of Illinois, Urbana, IL 61801, USA}
\email{ymb@illinois.edu}

\author[B.~Shapiro]{Boris Shapiro}
\address{Department of Mathematics, Stockholm University, SE-106 91
Stockholm,
         Sweden}
\email{shapiro@math.su.se}
\thanks{$^\dagger$Professor  
Nikolaus Braschmann ({\eightcyr Nikola}\!\!{\u {\eightcyr i}}
{\eightcyr  Dmitrievich Brashman}) was born in 1796 in a Jewish
merchant family in Neu Raussnitz of Austrian Empire (at present
Rousínov in Czech Republic). 
 Seldom remembered, he, as many other foreign scientists who worked in
 Russia in the 18th and the 19th centuries,  has substantially
 contributed to the progress of  Russian science and, in his
 particular case, to the development of   Moscow Imperial university
 (at present Moscow State university) and to the foundation of Moscow
 mathematical society as well as  of the Russian  journal
 Matematicheskii Sbornik. Judging from his portrait,  he has been
 decorated by the medal of St. Prince Vladimir ({\eightcyr orden
   svyatogo knyazya Vladimira, ``Vladimir na shee"}) which has been
 awarded for exceptional achievements in a military or a civil
 service.}

\date{\today}
\keywords{quadratic differentials,  Strebel differentials, non-chaotic differentials, signed measures}
\subjclass[2010]{Primary 30F30, Secondary 31A05}

\begin{abstract}   In this paper, motivated by the classical notion of
  a Strebel quadratic differential on a compact Riemann surface
  without boundary, we introduce several classes of  quadratic
  differentials (called non-chaotic, gradient, and positive gradient)
  which possess some properties of Strebel differentials and appear in
  applications. We discuss the relation between gradient
  differentials and special signed measures supported on their set of
  critical trajectories. We provide a characterisation of  gradient
  differentials for which there exists a positive  measure  in the
  latter class.
\end{abstract}
\dedicatory{\tencyri Professoru N.~D.~Brashmanu,$^\dagger$ osnovatelyu
  Matematicheskogo sbornika i uchitelyu P.L.~Chebysheva,
  posvyashchaet{}sya}

\maketitle

\section {Introduction}

Theory of quadratic differentials was pioneered in the late 1930's by
O.~Teichm\"uller as a useful tool to study conformal and
quasi-conformal maps.  Since then it has been substantially extended
and found numerous applications. (For the general information about
quadratic differentials consult \cite{Je, Stahl, Str}.) One important
class of quadratic differentials with especially nice properties was
introduced by J.~A.~Jenkins and K.~Strebel in the 50's; these
differentials are  called \emph{Strebel} or \emph{Jenkins-Strebel},
see \cite{Gar, Je,Str} and \S~\ref{sec:non-chaotic} below.   

In applications to potential theory, asymptotics of orthogonal
polynomials, WKB-methods in spectral theory of Schr\"odinger equations
in the complex domain one often encounters  quadratic differentials
which are non-Strebel, but rather share some of their properties, see
e.g. \cite{Ba, Fe, MFR1, MFR2, Sh, STT} and references therein. A good
example of such non-Strebel differentials having many important
properties  is provided by polynomial quadratic differentials in
$\bC$, such that each zero is the endpoint of some critical trajectory.

Motivated by the above examples, we present below several natural
classes of differentials containing the class of Strebel differentials
and possessing certain nice properties. The most general  class we
introduce is called \emph{non-chaotic} and it is characterized by the
property that the closure of any horizontal trajectory of such
differential is nowhere dense. Further we introduce a natural subclass of
non-chaotic differentials which we call \emph{gradient} and which is
characterised by the property that at its smoothness points it is
equal to $\pm \Re \sqrt{ \Psi}$. Finally, we discuss the appropriate
notion of positivity for gradient differentials. 

\begin{figure}
\begin{center}
\includegraphics[scale=2.5]{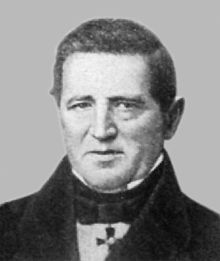}
\end{center}

\vskip 1cm

\caption{{\tencyr Nikola}\!\!{\u {\tencyr i}} {\tencyr Dmitrievich Brashman}.}
\label{Braschmann}

\end{figure}

\medskip
The structure of the paper is as follows. In \S~\ref{sec:quadr} we recall the basic facts about quadratic differentials and Strebel differentials.  
In \S~\ref{sec:non-chaotic}  we introduce and discuss a number of properties of non-chaotic differentials.  In \S~\ref{sec:gradient} we introduce and characterize gradient differentials.  In \S~\ref{sec:pos} we study positive gradient  differentials. 
Finally, in Appendix I 
 we recall our earlier motivating results   relating  considered classes of quadratic differentials to the Heine-Stieltjes theory, see \cite{STT}.

\medskip 
\noindent 
{\it Acknowledgements.}  The  second author wants  to acknowledge the hospitality of the Department of Mathematics of UIUC and the financial support of his visits to Urbana-Champaign under the program ``INSPIRE" without which this paper would never have seen the light of the day.

\section{Crash course on quadratic differentials} \label{sec:quadr}

\subsection{Basic notions}

The following definitions are  borrowed from  \cite {Je} and \cite {Str}.

\begin{defi+}
A (meromorphic) quadratic differential $\Psi$ on a compact  orientable Riemann surface $Y$ without boundary is
a (meromorphic) section of the tensor square
$(T^*_{\mathbb{C}}Y)^{\otimes2}$ of the holomorphic cotangent
bundle $T^*_{\mathbb{C}}Y$. The zeros and the poles of $\Psi$
constitute the set of \emph{critical points} of $\Psi$ denoted by
$Cr_{\Psi}$. (Non-critical points of $\Psi$ are  called {\em
regular}.) Zeros and simple poles are called {\em finite critical points} while poles of order at least $2$ are called {\em infinite critical points}. 
\end{defi+}

The next statement can be found in e.g.,  Lemma~3.2 of \cite{Je}.

\begin{lemma} \label{lm:euler} The Euler characteristic of $(T_\bC^*Y)^{\otimes 2}$ equals $2\chi(Y),$ where $\chi(Y)$ is the Euler characteristic of the underlying curve $Y$. Therefore, the difference between the number of poles and zeros (counted with multiplicity) of a meromorphic order $k$ differential $\Psi$ on $Y$ equals $2\chi(Y)$. In particular, the number of poles minus the number of zeros of any   quadratic differential $\Psi$ on $\bCP^1$ equals $4$. Such examples can be found in e.g. \cite{Fe}, Ch. 3. 
\end{lemma}

Obviously, if $\Psi$ is locally represented in two intersecting charts
by $f(z)dz^2$ and by $\tilde{f}(\tilde{z})d\tilde{z}^2$ resp. with a
transition function $\tilde{z}(z)$, then
$f(z)=\tilde{f}(\tilde{z})\left({d\tilde{z}}/{dz}\right)^2.$ Any
quadratic differential induces a metric on its Riemann surface $Y$
punctured at the poles of $\Psi$, whose length element in local
coordinates is given by
$$
|dw|=|f(z)|^{\frac{1}{2}}|dz|.
$$


The above canonical  metric $|dw|=|f(z)|^{\frac{1}{2}}|dz|$ on $Y$ is
closely related to  two distinguished line fields spanned by the
vectors $\xi\in T_zY$ such that
$f(z)dz^2$ is either positive or negative.  The integral curves of the  field given by $f(z)\xi^2>0$ are called \emph{horizontal
trajectories} of $\Psi$, while the integral curves of the second field  given by $f(z)\xi^2<0$ are called \emph{vertical trajectories} of $\Psi$. 
Trajectories of $\Psi$ can be naturally  parameterised by their 
arclength. In fact, in a neighbourhood of a regular point $z_0$ on
$Y,$ one can introduce a local coordinate $w$ called {\em canonical}Ê which is given by 
$$w(z):=\int_{z_0}^z\sqrt{f(\xi)}d\xi.$$ 
Obviously, in this coordinate the quadratic differential itself is given by  $dw^2=f(z)dz^2$ implying that 
horizontal trajectories on $Y$ correspond to horizontal straight lines in the $w$-plane. 

\medskip
Since we will only use horizontal trajectories of meromorphic
quadratic differentials, we will refer to them simply as 
\emph{trajectories}.

\begin{defi+}\label{def:domains}
A trajectory of a meromorphic quadratic differential $\Psi$ is called
\emph{critical}, if there exists a finite critical point of $\Psi$
belonging to its closure.  
For a
given meromorphic differential $\Psi$, denote by $K_\Psi\subset Y$ the
closure of the union of
critical trajectories of $\Psi$. 

\end{defi+}

Recall that, by Jenkins' Basic
Structure Theorem \cite[Theorem~3.5, pp. 38-39]{Je}, the set
$Y\setminus (K_\Psi \cup Cr_\Psi)$ consists of a finite number of the so-called 
circle, ring,
  strip and end domains. (For the detailed definitions and  information we refer  to loc. cit).  The names \emph{circle}, \emph{ring} and \emph{strip} domain are
describing their images under the analytic continuation of the mapping
given by the canonical coordinate; the \emph{end} domain (also referred to as \emph{half-plane} domain)  is mapped by the canonical coordinate onto the half-plane.

The interior of the $K_\Psi$ can be non-empty,
and consists of a finitely many components, each bounded by a (finite) union of critical trajectories. These components are referred to as the \emph {density} domains.

The decomposition of $Y\setminus(K_\Psi \cup Cr_\Psi)$ into circle, ring, strip, end and density domains
constitutes the so-called  \emph{domain configuration} of $\Psi$.


\medskip
It is known that quadratic differentials on $\bCP^1$ with at most
three distinct poles do not have density domains, see Theorem
3.6 (three pole theorem) of \cite{Je}.  This result, in particular, explains Example~\ref{ex1} below in which case the domain configuration consists only of strip and end domains, see e.g. \cite{Ba}.  But starting with 4 distinct poles in $\bCP^1,$ they
become unavoidable.


\subsection{Strebel differentials}

\begin{defi+}
A compact non-critical trajectory $\gamma$ of a meromorphic $\Psi$ is
called \emph{closed}. It  is necessarily diffeomorphic to a circle.
\end{defi+}

\begin{defi+}
A quadratic differential $\Psi$ on a compact Riemann surface $Y$ without boundary is called
\emph{Strebel}  if the complement to the union of its closed trajectories has vanishing Lebesgue measure.
\end{defi+}

\begin{rema+} In the nomenclature of Definition \ref{def:domains},
the complement
$Y\setminus (K_\Psi \cup Cr_\Psi)$ for an arbitrary Strebel differential $\Psi$
on $Y$ consists of (finitely many) circular and ring
domains, as can be easily deduced from the results of Ch. 3,
\cite{Str}.
\end{rema+}

One can also easily derive the following statement:

\begin{lemma}
\label{lemma1} If a meromorphic quadratic   differential $\Psi$
 is Strebel, then it has no poles of order
greater than 2. If it has a pole of order 2, then the residue of $\sqrt{\Psi}$ at
this pole is negative.
\end{lemma}




 \medskip
 These reasonings are summarized in the next statement. 
 
 \begin{lemma}\label{lm:structure} For any Strebel differential  $\Psi$ on~$Y$, the following holds. 

\noindent
{\rm (i)}  $K_\Psi$ is  the set  of all non-closed horizontal trajectories of $Y$  and $Y\setminus (K_\Psi\cup Cr_\Psi)$ 
is a disjoint union of finitely many cylinders.  

\noindent
{\rm (ii)} The metric $|\Psi|$ restricted to any of these  cylinders  gives the
standard metric of a cylinder with some perimeter~$p$ given by the length of the horizontal trajectories and
some length~$l$ given by  the length of the vertical trajectories
joining the bases of the cylinder. (Notice that $l$ can  be infinite.)

\noindent
{\rm (iii)}
A cylinder is conformally equivalent to the
annulus $e^{-l/p} < |z| < 1$, or to a punctured disc
if $l = \infty$.
\end{lemma}

\medskip

Strebel differentials play important role in the theory of univalent functions and the moduli spaces of algebraic curves. They enjoy a large number of extremal properties. Basic results on their existence and uniqueness can be found in Ch. VI of \cite{Str}, see especially Theorem 21.1. 





\section{Non-chaotic quadratic differentials}\label{sec:non-chaotic}

\begin{defi+} Given a meromorphic quadratic differential $\Psi$ on a
  compact Riemann surface $Y$, we say that $\Psi$ is {\it non-chaotic}
  if there exists a continuous and piecewise smooth function 
$$F:Y\setminus Cr_\Psi \to \bR$$ defined
  on the complement to the set $Cr_\Psi $ of critical points of  $\Psi$ such that:  

\medskip
\noindent
(i) $F$ is non-constant on any open subset of $Y\setminus Cr_\Psi$,

\noindent
(ii) yet $F$ is constant on each horizontal trajectory of $\Psi$.

Such a function $F$ is called a {\it level function} of $\Psi$.
\end{defi+}

\begin{example}\label{ex1} A polynomial quadratic differential
  $\Psi=P(z)dz^2$, (where $P(z)$ is a univariate polynomial) is non-chaotic on $\bC P^1$.
\end{example}

It is almost immediate that Strebel differentials are non-chaotic: the
distance (in the Riemannian metric induced by $\Psi$) to $K_\Psi\cup Cr_\Psi$
serves as the level function.

\medskip
Non-chaotic quadratic differentials are easy to characterise in terms
of domain decompositions. Namely, the following statement holds.  

\begin{lemma}\label{lm:non-chaotic} A quadratic differential $\Psi$ is
  non-chaotic if and only if $\Psi$ has no density domains, i.e., the
  closure of each horizontal trajectory of
$\Psi$ on $Y$ coincides with this trajectory plus possibly some critical points. 
\end{lemma}

\begin{proof}
Assume non-chaoticity. In this case $K_\Psi\cup Cr_\Psi$ is a union of finite
number of critical
trajectories and critical points, and its complement is a union of domains comprised
either of compact trajectories (ring and circle domains) or of trajectories
isometric to real line (strip and end domains). On each of these
domains we can construct a function that is continuous, constant on
the trajectories, but not on any open set, and which is vanishing on
the boundary of the domain (e.g. by taking the sine of the imaginary part of the
canonical coordinate). Gluing together these functions (originally defined on individual domains, but vanishing on the boundary) 
along $K_\Psi$ delivers the desired continuous  level function.

If $\Psi$ {\em is chaotic}, there exists a 
trajectory with closure having a non-empty interior. A level function
should be constant on this trajectory and continuous, hence constant
on an open set, a contradiction.
\end{proof}

As we mentioned, any Strebel differential is non-chaotic. Moreover, the following holds.  

\begin{proposition}\label{pr:Strebel}
A quadratic differential $\Psi$ is Strebel if and only if it is
non-chaotic and  has a level function $F$ with finite limits at each
critical point $p\in Cr_\Psi$, i.e. a level function $F$ that can be extended by continuity from $Y\setminus Cr_\Psi$ to  $Y$.
\end{proposition}

\begin{proof}
As we mentioned above, a non-chaotic differential is Strebel if and
only if its domain decomposition consists only of ring and circle
domains. It is clear that in this case the construction of Lemma
\ref{lm:non-chaotic} yields a function continuous on all of
$Y$. Conversely, the existence of an  end or a strip domain implies that there is a
one-parametric family of non-critical trajectories
converging (on one end of the strip) to a critical point $C_o\in Cr_\Psi$. The union of
these trajectories forms a sub-strip in a strip or an end domain. A
non-constant function on such a sub-strip which is constant along the
trajectories will automatically have discontinuity at $C_o$.
\end{proof}


To move further, let us recall some basic facts from complex analysis and
potential theory  on  Riemann surfaces, see e.g. \cite {GH}.

Let~$Y$ be
an  (open or closed) Riemann surface and $h$ be a real- or
complex-valued smooth  function on~$Y$.  

\begin{defi+}
The {\em Levy form} of $h$ (with respect to a local coordinate 
$z$) is given by 
\begin{equation}\label{levyn} 
\mu_h:=2i \frac{\partial^2 h}{\partial z \partial \bar z}\; dz\wedge d\bar z.
\end{equation}
\end{defi+}

In terms of the real and imaginary parts $(x,y)$ of  $z,$  $\mu_h$ is given by 
$$\mu_h=\left(  \frac{\d^2 h}{\d x^2}+ \frac{\d^2 h}{\d y^2} \right) \;  dx \wedge dy=\Delta h dx \wedge dy. $$

If $h$ is a smooth real-valued function, $\mu_h$ can be also thought
of as a real signed measure on $Y$ with  a smooth density. In
potential theory $h$ is usually referred to as the {\em (logarithmic)
  potential}  of the measure $\mu_h$, see e.g. \cite{GH}, Ch.3. Notice
that (\ref{levyn}) makes sense for  an arbitrary  complex-valued
distribution $h$ on~$Y$ if one interprets $\mu_h$  as a 2-current on
$Y$, i.e. a linear functional on the space of smooth compactly
supported functions on $Y$,  see e.g. \cite {Fed}.

Such a current is necessarily exact since the inclusion of
smooth forms into currents induces the (co)homology isomorphism.
Recall that any complex-valued measure on $Y$ is a $2$-current
characterised by the additional requirement that its value on a smooth
compactly supported function depends only on the values of this
function, i.e. on its $0$-jet (and does not depend on its derivatives,
i.e. on higher jets).  Notice that if $Y$ is compact and connected,
then exactness of $\mu_h$ is equivalent to the vanishing of the
integral of $\mu_h$ over $Y$.

We should remark here that the Levy form depends on the (local) metric
structure defined by the (local) coordinate $z$, unless it is a sum of
the delta-functions.

\begin{example} \label{Ex:1} a) If $h= \ln |z|$ on $\CP^1$, then $\mu_h$ is a real measure
supported on the $2$-point set $\{0, \infty\}$ with $\mu_h(0) = 2\pi$, $\mu_h(\infty) = -2\pi$.  

b) If $h = |\Im z|$ on $\bC$, then $\mu_h$ is a measure
supported on the real axis, namely, $\mu_h = 2 dx$, i.e. twice the usual Lebesgue measure on the real line. 
\end{example}

The easiest way to verify these examples is to use Green's formula:
$$
\mu_h(D) = \int_D 2i\;\d \db h = \int_{\d D} \frac{\d h}{\d n} dl,
$$
which provides a way to calculate $\mu_h(D)$ where $D$ is a arbitrary compact domain in $Y$ with a smooth boundary; $ \frac{\d h}{\d n}$ is the derivative of $h$ w.r.t the outer normal, and $dl$ is the length element of the boundary.

$$
\int_D f \Delta h = - \int_D (df, dh) + 
\int_{\d D} f \frac{\d h}{\d n} dl.
$$
Here $D$ is a domain bounded 
by a piecewise smooth loop $\d D$ orinted counterclockwise; $\d h/ \d n$ is the
derivative in the orthogonal direction to $\d D$, while
$dl$ is the length element on $\d D$.





\medskip
Our next goal is, for a given non-chaotic $\Psi$,   to find its level function $F$  which is closely related to the metrics on $Y$ induces by $\Psi$ or, alternatively, whose Levi  form $\mu_F$ has a small support. 
Such $F$ is readily available by the following statement. 

\begin{proposition}\label{prop:prepot} For any non-chaotic differential $\Psi$ on $Y$, there exist its level functions which are piecewise harmonic and which are non-smooth on finitely many trajectories of $\Psi$. 
\end{proposition} 

\begin{proof}
Take as $F$ the {\em distance to $K_\Psi \cup Cr_\Psi$} in the metric defined by
$\Psi$. Locally, its differential coincides with $\Im\sqrt{\Psi}$,
hence is harmonic. It is immediate that $F$ is smooth outside $K_\Psi\cup Cr_\Psi$ and
the (finite) union of the horizontal trajectories running in the
middle of the strip and annular domains.
\end{proof}

\begin{rema+} We will call level functions constructed in Proposition~\ref{prop:prepot} \emph{piecewise harmonic}. For any piecewise harmonic level function $F$, its Levy form $\mu_F$ is a well-defined  $2$-current  supported  on the   union of $Cr_\Psi$ and finitely many trajectories where $F$ is non-smooth. For example, in case of a Strebel differential $\Psi$, the Levy form of any piecewise harmonic level function will have point masses exactly at the double poles of  $\Psi$. 
\end{rema+}

\def\tY{\tilde{Y}}

\section {Gradient differentials}\label{sec:gradient}
\def\sPsi{\sqrt{\Psi}}
\def\bKP{\bar{K}_\Psi}
\def\bK{\bar{K}}
\def\rg{\mathit{Rg}}

Denote by $\bKP$ the union of the critical trajectories and the zeros
and simple poles of $\Psi$, 
$$
\bKP=K_\Psi\cup Cr^{\mathtt{f}}\Psi.
$$

This is a  one-dimensional cell complex embedded into $Y \setminus Cr^-_\Psi$, where $Cr^-_\Psi$ is the union of all infinite critical points and equipped with a
metric (given by $|d\Im\sPsi|$) that turns $\bKP$ into a complete
metric space (the topologically open ends of the graph have an
infinite length). 

Consider the decomposition 
$$
\bKP=\amalg_{\alpha\in A} \bK_\alpha
$$
into the set of its connected components.

Each of these components also carries  the structure of the
{\em fat graph}, encoded by the collection of cyclic permutations of the
edges incident to a given vertex, one
for each vertex of the graph. 

These cyclic permutations can be thought of as a single permutation
$\sigma_0$ of the
set of {\em flags of a graph}, that is of the pairs consisting of a vertex and
its incident edge; the orbits of the permutation $\sigma_0$ are in
one-to-one correspondence with the vertices of the fat graph. 

The other permutation of the flags of the fat graph is the involution
$\sigma_1$ interchanging the two flags corresponding to the same edge. 

The composition cycles of the product $\sigma_0\sigma_1$
correspond to the {\em boundary components} of the fat graph, which
list  the trajectories bounding
the connected components of the complement $Y \setminus K_\Psi$ in the order fixed by the orientation.

\medskip
We define the {\em Reeb graph} $\rg_\Psi$ of a non-chaotic quadratic
differential $\Psi$ as follows. 
\begin{Def}
  {\rm The \emph{Reeb graph} $\rg_\Psi$ is the metric graph with possibly
   edges of infinite length necessarily ending at leaves. The vertices
  $V_{\Psi}=A$ of the Reeb graph are identified with the set of
  connected components of $\bKP$. The edges $E_{\Psi}$ are
  the spaces of noncritical trajectories (or, equivalently, the factor
  spaces of the connected components of $Y\setminus K_\Psi$ by the equivalence
  relations given by belonging to the same trajectory). The lengths on
  the edges are given by $|\int\Im\sPsi|$.}
\end{Def}

The Reeb graph $\rg_\Psi$ might have loops and multiple edges.
  We remark that on each of the connected components of $Y\setminus K_\psi$,
the {\em square root} of the quadratic differential is the meromorphic
1-form $\sPsi$ defined unambiguously, up to a sign.

\smallskip
The lengths of the edges are finite for the strip or ring components,
and infinite for the circle or end components. The
components corresponding to edges of infinite length are necessarily adjacent to
the poles of order at least $2$, i.e. to the infinite critical points.

We will call a level function $F$ {\em natural} if on any of the
connected components of $Y\setminus K_\Psi$, its gradient matches the real part
of a branch of $\sPsi$:
$$
dF=\pm\Re\sPsi.
$$

A natural level function fixes the orientation on each of the edges of
the Reeb graph $\rg_\Psi$. Together with the length elements on the
edges, these orientations define a family of 1-forms on the edges, and
hence a de Rham cocycle on the Reeb graph $\rg_\Psi$.

\begin{proposition}\label{prop:gradient}
A non-chaotic differential  $\Psi$ admits a natural level function if
and only if the edges of $\rg_\Psi$ can be oriented in such a way that
the resulting $1$-cocycle on $\rg_\Psi$ is trivial. In other words, 
 the sum of the lengths of the edges in any oriented cycle in the
Reeb graph, taken with the signs $\pm$ depending on whether  the
orientation of the cycle is consistent with the orientations of the
edges or not, vanishes.

Conversely, any such orientation defines a natural level function up
to an additive constant.
\end{proposition}

\begin{defi+}
Any non-chaotic quadratic differential satisfying the conditions of
Proposition \ref{prop:gradient} will be called \emph{gradient}, and  
any of the corresponding level functions $F$ will be called a {\em potential}.
\end{defi+}

Any potential of a gradient quadratic differential is constant on the
components of $\bKP$. 

\begin{proof}[Proof of Proposition \ref{prop:gradient}]
The claim that a potential defines an orientation on the edges of the Reeb graph
is immediate from the definition, as is the exactness of that cocycle.
Conversely, the exactness of the cocycle on the Reeb graph defined by the length
elements and orientations on the  edges allows one to integrate it to a
function on the Reeb graph, which lifts to a potential.
\end{proof}

\begin{lemma}\label{lm:grad} 
Levy form $\mu_F$ of any potential $F$ of a gradient quadratic differential $\Psi$ is supported on $K_\Psi \cup Cr_\Psi$.
\end{lemma}

\begin{proof}
Indeed, the restriction of the potential to each of the domains in
$Y\setminus (K_\Psi\cup Cr_\Psi)$ is harmonic.
\end{proof}

The potential functions may fail to exist (for example, if the Reeb
graph has a loop). But their number is obviously finite (as we
can identify the potential functions with an element of the finite
set of orientations of the edges of the Reeb graph). In fact, more can
be said:

\begin{prop+}\label{prop:2^n} For a gradient differential $\Psi$, the
  number of different potentials (considered up to an additive
  constant) is either $0$ or a power of 2, comp. Theorem~4 of \cite{STT}. 
\end{prop+} 

\begin{proof}
  The group $\mathtt{Flips}=\Int_2^{E_\Psi}$ of flipping the
  orientations of the edges acts on the space of cochains on the Reeb
  graph by reflections. The collections of flips that preserve the
  subspace annihilating the cycles in the Reeb graph is, clearly, a
  subgroup in $\mathtt{Flips}$.
\end{proof}

Existence of the potential function poses further restrictions on the
local properties of the quadratic differential $\Psi$.

\medskip
Let $F$ be a potential for $\Psi$. We will refer to a pole of $\Psi$
as {\em $F$-clean} if it does not belong to the support of the Levy
measure $\mu_F$. 

\begin{lemma}
The order $r\geq 2$ of any $F$-clean pole is even.
\end{lemma}
\begin{proof}
Indeed, the $\Int_2$-bundle of orientations defined by $\pm dF$ does
not admit a section in a (punctured) vicinity of a pole of $\Psi$ of
odd order. Hence $dF$ is discontinuous in an arbitrarily small neighbourhood  of the pole.  
\end{proof}

The $F$-clean poles of even order  exist. Moreover,  
\begin{prop+}
Let $F$ be a potential for a non-chaotic quadratic differential
$\Psi$. Then for  any $F$-clean pole $z_*$ of $\Psi$ of even order, 
 the residue of the $\sPsi$ (defined up to a sign near $z_*$) is
purely imaginary. 
\end{prop+}
\begin{proof}
The statement follows immediately from the fact that for a clean pole,
$F$ is smooth in a punctured vicinity of $z_*$, and therefore the
increment of the potential $F$ equals the residue of
$\sPsi$.
\end{proof}

\begin{lemma}\label{lm:restore}
A gradient differential $\Psi$ on a compact $Y$ is uniquely defined by its Levy form $\mu_F$ of any of its potentials $F.$
\end{lemma}

\begin{proof}
Two functions $F_1$ and $F_2$ (considered as 0-currents) have the same
Levy forms (considered as 2-currents) only if the difference
$F_1-F_2$ has vanishing Laplacian, and hence, by compactness of $Y$,
is a constant.

Now, if two gradient quadratic differentials $\Psi_1, \Psi_2$ have
corresponding potentials $F_1, F_2$ coinciding (up to a constant) on
an open subset of $Y$, then the (locally defined) holomorphic 1-forms
$\sPsi_1, \sPsi_2$ have identical real parts (equal to $dF_1, dF_2$,
respectively) on the same subset, and hence coincide everywhere.

\end{proof}

\section{Levy measures and Positivity}\label{sec:pos}
In this section we discuss the notion of positivity
for  gradient  quadratic differentials. Observe that for any potential
function $F$, its Levy form $\mu_F$ is an exact $2$-current on $Y$,
i.e., $\int_Y\mu_F=0$. Many applications in asymptotic analysis lead
to the situation when a gradient
differential has a potential $F$ whose Levy form is a signed measure whose positive part is 
supported on $K_\Psi$, and whose negative part is supported on ($F$-clean)
poles of $\Psi$. (We discuss an example in \S~\ref{sec:HS}.)   

\begin{defi+} We will call {\em positive} a clean potential $F$ such that
  the restriction of $\mu_F$ to $K_\Psi$ is a positive
  measure. A quadratic potential admitting a positive potential will
  also be referred to as positive. 
\end{defi+}
We remark that the notion of positivity depends only on the potential
function $F$ but not on the particular coordinate chart.

Whether or not a potential of a gradient quadratic differential $\Psi$ is positive,
depends not only on its Reeb graph, but also on the structure of fat graphs
 for the components $K_\alpha$ corresponding to the vertices
of the Reeb graph.

Specifically, each edge $e^\alpha_\beta$ of the fat graph $K_\alpha$
is adjacent to one or two boundary components (corresponding to the
orbits of the two flags incident to the edge under the action of permutation 
$\sigma_0\sigma_1$ defining the fat graph structure). The boundary
components of the fat graph $K_\alpha$ correspond to the edges of
$\rg_\Psi$. We will be calling these edges of
the Reeb graph \emph{incident to the corresponding edge} $e^\alpha_\beta$ of
the fat graph $K_\alpha$.

\begin{lemm+}\label{lem:positive}
The potential $F$ of a gradient quadratic differential is positive if and only if 
for any edge $e^\alpha_\beta$ of any of the fat graphs $K_\alpha$,
the orientation of the Reeb graph defined by $F$ has at least one of
the (at most two) incident edges of the Reeb graph  oriented
towards $K_\alpha$. 
\end{lemm+}
\begin{proof}
An immediate local computation shows that if the edges incident to
$e_\beta^\alpha$ are both oriented away from $K_\alpha$, the
corresponding measure on the edge equals $-2|d\Im\sPsi|$; if one edge is
oriented towards, and one away from $K_\alpha$, then $\mu_F=0$ near
the edge, and if incident edges are oriented towards  $K_\alpha$, then $\mu_F=2|d\Im\sPsi|$.
\end{proof}

\begin{figure}[htp]
\begin{center}
\includegraphics[width=.7\textwidth]{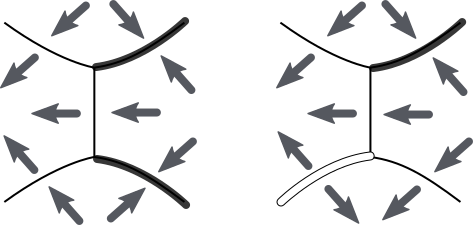}
\caption{The figure showing how the orientation of the {\em Reeb}
  graph affects the positivity on a component of $K_\Psi$: on the left
  display, all edges of the fat graph component have the Reeb graph
  orientation pointing toward them, or ``through'' them; on the right
  display, the SW edge has both adjacent edges of the Reeb graph
  outgoing, resulting in locally negative mass. (Positive charges are
  shown as fat solid lines, negative, - as fat hollow line.)}
\label{fig:orientation}
\end{center}
\end{figure}

We remark that  Lemma \ref{lem:positive} turns the computational
question of the positivity of a given gradient quadratic differential
$\Psi$ into an instance of a $2$-satisfiability problem, \cite{moore}. Indeed, one can
interpret the orientations of the edges of the Reeb graph as
Boolean variables, and the absence of two outgoing edges of the Reeb
graph incident to an edge of a fat graph $K_\alpha$ as a
$2$-clause. Such interpretation implies that given the fat graph
structures, the positivity  can be efficiently decided  in time quadratic in the number of
the critical points of $\Psi$.

The natural length element $|d\Im\sPsi|$ on the edges of the fat
graphs $K_\alpha$'s (or critical trajectories of $\Psi$) defines the
{\em widths} on the boundary components of the fat graphs $K_\alpha$,
or, equivalently, on the edges of the Reeb graph $\rg_\Psi.$ (Recall
that the
{\em lengths} of the edges of the Reeb graph are defined by
$|d\Re\sPsi|$.) For the components
containing poles of $\Psi$ in their closure, the width can be infinite.

\medskip
Next result is immediate:

\begin{lemm+}\label{lm:mass}
The total mass of $\mu_F$ supported by a component $K_\alpha$ equals 
to the difference of the widths of all incoming and all outgoing components.
\end{lemm+}

Lemma~\ref{lm:mass} implies the following necessary condition of positivity:
\begin{coro+}
If a potential $F$ is positive, then for any component $K_\alpha$, the total
width of the incoming edges is greater than or equal to the total width of the
outgoing edges.
\end{coro+}

It is worth mentioning  that the latter condition is not sufficient: just
fixing the Reeb graph of a gradient quadratic differential, and widths and orientations of its edges is
not enough to determine the positivity of the corresponding potential. Indeed, 
the Dehn twists acting on the space of the corresponding complex
structures of the Riemann surface $Y$ do not preserve positivity,  see
Fig.~\ref{fig:nopreserve}.

\begin{figure}[htp]
\begin{center}
\includegraphics[width=.9\textwidth]{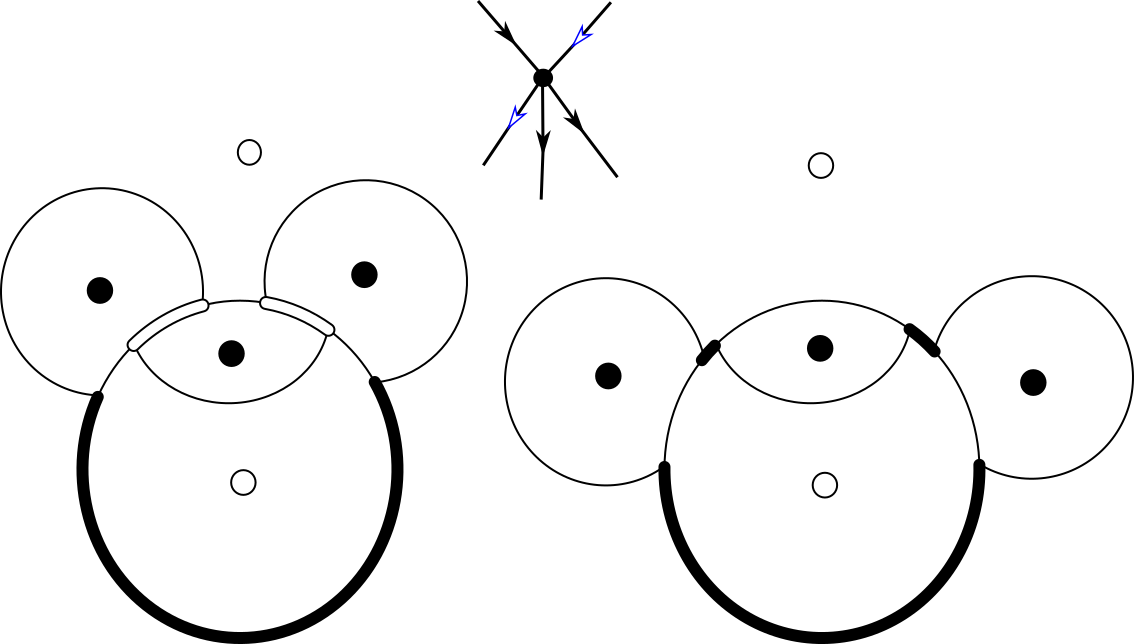}
\caption{Singular sets and trajectories of clean quadratic differ-
  entials on $\CP^1$ with 5 poles of order 2, and 6 simple zeros each. The
  residues of the poles (equal to the widths of the circle domains
  centered at these poles) are the same for corresponding poles on the
  left and on the right pictures. Yet the potentials on the left has
  negative components of $\mu_F$; the one on the right is
  positive. The Reeb graph is sketched on the top. (We keep the
  convention that positive masses are shown as solid lines or dots;
  negative, - as hollow ones.)}
\label{fig:nopreserve}
\end{center}
\end{figure}

Another constraint on the orientation of the edges of the Reeb graph
required by the positivity of a potential comes from the {\em simple
  poles} of $\Psi$, see Fig.~\ref{fig:pole}. As the edge of the fat graph adjacent to a simple
pole has the same domain on both sides, the positivity implies that
the orientation of the edge of the Reeb graph should be {\em
  incoming}, comp. Proposition~2, \cite{STT}.

\begin{figure}[htp]
\begin{center}
\includegraphics[width=1in]{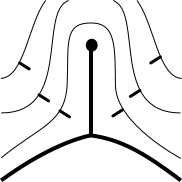}
\caption{Positivity forces the orientation of the potential function
  near a simple pole of $\Psi$.}
\label{fig:pole}
\end{center}
\end{figure}

\section{Appendix I. Quadratic differentials and Heine-Steiltjes theory}\label{sec:HS}

We have earlier encountered Strebel and gradient differentials  in the study of the  asymptotic properties of Van Vleck and Heine-Stieltjes polynomials and solutions of Schr\"odinger equation with polynomial potential, see \cite{HoSh, STT, ShQ}.  Some of these results are presented below and they were a major motivation for the present study.

\medskip
Given a pair of polynomials $P(z)$ and $Q(z)$ of degree $m$ and at most $m-1$ respectively,   consider the differential equation:
\begin{equation}\label{eq:HS}
P(z)S^{\prime\prime}(z)+Q(z)S'(z)+V(z)S(z)=0.
\end{equation}

\medskip
The classical Heine-Stieltjes problem for  equation (\ref{eq:HS}) asks for any positive integer $n,$ to find the set of all possible polynomials $V(z)$ of degree at most $m-2$ such that  (\ref{eq:HS})  has a polynomial solution $S(z)$ of degree $n$, see \cite{He}, \cite{Sti}.  Already  E.~Heine proved that for a generic equation   (\ref{eq:HS}) and any positive $n,$ there exist $\binom {n+l-2}{l-2}$ polynomials $V(z)$ of degree $l-2$ having the corresponding polynomial solution $S(z)$ of degree $n$. Such polynomials $V(z)$ and $S(z)$ are referred to as {\em Van Vleck} and {\em Heine-Stieltjes} polynomials respectively.  The following localization result for the zero loci of $S(z)$ and $V(z)$ was proven in  \cite{Sh}. 

\begin{proposition}\label{th:loc} For any $\epsilon >0,$ there exists $N_\epsilon$ such that all roots of 
$V(z)$ and its corresponding $S(z)$ lie within  $\epsilon$-neighbourhood of $Conv_P$ if $\deg S(z)\ge N_\epsilon$. Here $Conv_P$ stands for the convex hull of the zero locus of the leading coefficient $P(z)$.
\end{proposition}

The above localization result implies that there exist plenty of converging subsequences $\{\tilde V_n(z)\}$ where $V_n(z)$ is some Van Vleck polynomial for equation (\ref{eq:HS})  whose Stieltjes polynomial $S_n(z)$ has degree $n$ and $\tilde V_n(z)$ is the  monic polynomial proportional  to $V_n(z)$. (Convergence is understood coefficient-wise.) 

Recall that   the Cauchy transform $\C_\nu(z)$ and the logarithmic potential $u_\nu(z)$ of a (complex-valued) measure $\nu$ supported in $\bC$ are by definition given by: 
$$\C_\nu(z)=\int_{\bC}\frac{d\nu(\xi)}{z-\xi}\quad\text{    and    }\quad u_\nu(z)=\int_{\bC}\log|z-\xi|{d\nu(\xi)}.$$
Obviously, $\C_\nu(z)$ is analytic outside the support of $\nu$ and
has a number of important properties, e.g. that
$$\C_\nu(z)=\frac{\partial u_\nu(z)}{\partial z};\quad\quad \nu=\frac{1}{\pi}\frac{\partial\C_\nu(z)}{\partial \bar z}$$ where the derivative is understood in
the distributional sense.  Detailed information about Cauchy transforms
can be found in e.g. \cite{Ga}.

\begin{theorem}\label{th:asym} Choose a sequence $\{V_n(z)\}$  of Van Vleck polynomials  ï¿½
where $\deg S_n(z)=n$ with converging sequence $\{\tilde V_n(z)\}\to \tilde V(z)$.  Then the sequence of root-counting measures $\mu_n$ of $S_n(z)$ weakly converges to the probability measure $\mu$ whose Cauchy transform $\CC_\mu(z)$ satisfies a.e. in $\bC$ the algebraic equation 
$$\CC_\mu^2(z)=\frac{\tilde V(z)}{P(z)}.$$
Moreover, the logarithmic potential $u_\mu(z)$ of $\mu$ has the property that its set of level curves coincides with the set of closed trajectories of the quadratic differential $-\frac{\tilde V(z)dz^2}{P(z)}$ which is therefore Strebel.
\end{theorem}

The Theorem~\ref{th:asym} implies further results for arbitrary
rational Strebel differentials with a second order pole at
$\infty$. (These statements are special cases of the results in \S~\ref{sec:pos}.)

\medskip
\begin{theorem}[see  Theorem~4, \cite {STT}] \label{th:charac} Let $U_1(z)$ and $U_2(z)$ be arbitrary  monic complex polynomials
with $\deg U_2 -\deg U_1=2$. Then 
\begin{enumerate}
\item
the  rational quadratic differential
$\Psi=-{U_1(z)}dz^2/{U_2(z)}$ on $\bCP^1$ is Strebel if and only if   
there exists a real and compactly supported in $\bC$ measure
$\mu$ of total mass $1$ (i.e. $\int_{\bC}d\mu=1$)
 whose Cauchy transform $\C_{\mu}$ satisfies a.e. in $\bC$ the equation: 
\begin{equation}\label{eq:imp}
\C_{\mu}^2(z)={U_1(z)}/{U_2(z)}.
\end{equation}

\item  for any $\Psi$ as in $\rm(1)$ there exists exactly $2^{d-1}$  real measures whose Cauchy transforms  satisfy \eqref{eq:imp} a.e. and whose support is contained in $K_\psi$ where $d$ is the total number of connected components in $\bCP^1\setminus K_\Psi$ (including the unbounded component, i.e. the one containing $\infty$). These measures are in $1-1$-correspondence with $2^{d-1}$ possible choices of the branches of $\sqrt{{U_1(z)}/{U_2(z)}}$  in the union of $(d-1)$ bounded components of $\bCP^1\setminus K_\Psi$. 

\end{enumerate} 
\end{theorem}

 Concerning  measures positive in $\bCP^1$, with the only non-simple
 pole at infinity, we notice first, that the Reeb graph is necessarily
 a tree, with infinite length edges corresponding to the edge domains
 (adjacent to $\infty\in\bCP^1$) and with the leaves which correspond
 to the components of $K_\Psi$ containing, necessarily, simple poles of
 $\Psi$. Given that, the following statement should be quite obvious.

\begin{theorem} [see  Theorem~5, \cite {STT}] \label{pr:positive} For any Strebel differential $\Psi=-{U_1(z)}dz^2/{U_2(z)}$ on $\bCP^1$  (in the notation of Theorem~\ref{th:charac}) there exists at most one positive measure 
satisfying \eqref{eq:imp} a.e. in $\bC$. Its support necessarily belongs to $K_\Psi$, and, therefore,  among  $2^{d-1}$ real measures  described in Theorem~\ref{th:charac} at most one is positive. 
\end{theorem}

Moreover,   we can formulate an exact criterion of the existence of a positive measure in terms of rather simple topological properties of $K_\Psi$.   To do this we need one more definition. Observe that in our situation $K_\Psi$ is a planar multigraph.


\begin{defi+} By a {\em simple cycle} in a planar multigraph $K_\Psi,$ we mean any closed non-self-intersecting curve formed by  the edges of $K_\Psi$. (Obviously, any simple cycle bounds an open domain homeomorphic to a disk which we call the {\em interior of the cycle}.) 
\end{defi+}

\begin{center}
\begin{picture}(440,160)(0,0)

\put (85,120){\circle{25}}

\put(80,120){\circle*{3}}
\put(80,120){\line(1,0){10}}
\put(90,120){\circle*{3}}

\put (125,120){\circle{25}}
\put(120,120){\circle*{3}}
\put(120,120){\line(1,0){10}}
\put(130,120){\circle*{3}}

\put(97,120){\circle*{3}}
\put(97,120){\line(1,0){16}}
\put(113,120){\circle*{3}}

\qbezier(60,120)(75,195)(150,120)
\qbezier(60,120)(75,45)(150,120)
\put(150,120){\circle*{3}}
\put(150,120){\line(1,0){20}}
\put(170,120){\circle*{3}}

\put(230,120){\line(1,0){60}}
\put(230,120){\circle*{3}}
\put(290,120){\circle*{3}}

\qbezier(260,145)(310,175)(290,120)
\qbezier(260,145)(210,175)(230,120)
\qbezier(260,85)(310,90)(290,120)
\qbezier(260,85)(210,90)(230,120)
\put(260,145){\circle*{3}}
\put(260,145){\line(0,-1){12}}
\put(260,133){\circle*{3}}

\put(280,130){\circle*{3}}
\put(240,130){\circle*{3}}
\qbezier(240,130)(260,122)(280,130)

\put(240,105){\circle*{3}}
\put(280,105){\circle*{3}}
\qbezier(240,105)(260,95)(280,105)





\put(20,40){Figure 5. $K_\Psi$ admitting and not admitting a positive measure}
\end{picture}
\end{center}

\begin{Prop} [see  Proposition~2, \cite {STT}] \label{pr:critQ} A Strebel differential $\Psi=-{U_1(z)}dz^2/{U_2(z)}$ admits a positive measure satisfying \eqref{eq:imp} if and only if no edge of $K_\Psi$ is attached to a simple cycle from inside. In other words,  for any simple cycle in $K_\Psi$ and any edge not in the cycle but adjacent to some vertex in the cycle 
this edge does not belong to its interior. The support of the positive measure coincides with the forest obtained from $K_\Psi$ after the removal of all its simple cycles.  
 \end{Prop}

\end{document}